\documentclass[twoside,11pt]{article}
\usepackage{amsmath,amssymb,mathrsfs}
\newtheorem{thm}{Theorem}[section]

\newcommand{\ee}{\end{equation}}
\newenvironment {proof} {\noindent {\bf Proof.}}{\quad $\square$\par\vspace{3mm}}

\newtheorem{dif}{Definition}[section]
\newtheorem{lema}{Lemma}[section]
\newtheorem{remark}{Remark}[section]

\usepackage{amsmath,amsfonts}

\allowdisplaybreaks[4]

\begin{document}
\newcommand{\D}{\displaystyle}

\title{The existence and blow-up criterion of liquid crystals system in critical Besov space
\footnotetext[0]{This work was supported partly by NSFC grant 11071043, 11131005, 11071069.}
}
\author{Yi-hang Hao\footnote{{\it E-mail address}: 10110180022@fudan.edu.cn.}\qquad Xian-gao Liu\footnote{{\it E-mail address}: xgliu@fudan.edu.cn,}
\\{ \it\small  School of Mathematic Sciences, Fudan University,
Shanghai, 200433, P.R.China }}
\date{}

\maketitle
\begin{abstract}
We consider the existence of strong solution to liquid crystals system in critical Besov space, and give a criterion which is similar to Serrin's criterion on regularity of weak solution to Navier-Stokes equations.
\end{abstract}

\date{}

 \vspace{3mm}
\maketitle


\section {Introduction}
In this paper, we study the following incompressible liquid crystals system:
\begin{eqnarray}
\left\{\begin{array}{ll}
& \mbox{{\bf div}}u=0, \\
&u_t+u\cdot\nabla u-\mu\Delta u+\nabla p=-\nabla(\nabla d\odot\nabla d),\\
&d_t+ u\cdot\nabla d-\Delta d-|\nabla d|^2d=0,
\end{array}\right.\label{prob1}
\end{eqnarray}
on $(0,T)\times R^\mathbb{N}$. Here $u$ is the velocity, $p$ is the pressure, $d$ is the macro-scopic average of molecular arrangement, $\nabla d\odot\nabla d$ is a matrix, whose $(i,j)-\text{th}$ entry  is $d_{k,i}d_{k,j}$. Let's consider system (\ref{prob1}) with the following initial conditions
\begin{eqnarray}
u|_{t=0}=u_0(x),\quad d|_{t=0}=d_0(x),\quad |d_0(x)|=1,\label{ini}
\end{eqnarray}
and far field behaviors
\begin{eqnarray}
u\rightarrow 0,\ d\rightarrow \bar{d_0},\ as\ |x|\rightarrow\infty.\label{ini1}
\end{eqnarray}
Here $\bar{d_0}$ is a constant vector with $|\bar{d_0}|=1$.
As it is difficult to deal with the high order term $|\nabla d|^2d$, Lin-Liu \cite{LinLiu1,LinLiu2} proposed to investigate an approximate model of liquid crystals system by Ginzburg-Landau function. The approximate liquid crystals system reads as follows
\begin{eqnarray}
\left\{\begin{array}{ll}
& \mbox{{\bf div}}u=0, \\
&u_t+u\cdot\nabla u-\mu\Delta u+\nabla p=-\nabla(\nabla d\odot\nabla d),\\
&d_t+ u\cdot\nabla d-\Delta d+\frac{1}{\epsilon^2}(1-|d|^2)d=0.
\end{array}\right.\label{approxi}
\end{eqnarray}
\par The hydrodynamic theory of nematic liquid crystals was established by Ericksen \cite{eri} and Lieslie \cite{les} in 1950s. Lin \cite{Linfh} introduced a simple model and studied it from mathematical viewpoint. Since then, many applauding results to system (\ref{approxi}) with (\ref{ini}) and (\ref{ini1}) have been established: {\bf1.} the global existence of weak solutions, see \cite{LinLiu2,LinLiu3}; {\bf2.} the global existence of classical solution for small initial data or large viscosity, see \cite{LinLiu3,WuXuLiu}; the local existence of classical solution for general initial data, see \cite{LinLiu3}; {\bf3.} the partial regularity property of the weak solutions, see \cite{LinLiu1}; {\bf4.} the stability and long time behavior of strong solution, see \cite{WuXuLiu}.
 \par From the viewpoint of partial differential equations, system (\ref{prob1}) is a strongly coupled system between the incompressible Navier-Stokes equations and the transported heat flow of harmonic maps. It
 is difficult to get the global existence for general initial data, even the existence of weak solutions in three-dimensional space. However, Lin et al\cite{cywang} and Hong\cite{hmc} obtained the global weak solutions in two-dimensional space. In this paper, a solution is called to be a strong solution if the uniqueness holds. The well-posedness  of system  (\ref{prob1})-(\ref{ini1}) has been studied by Li and Wang\cite{Liwang}, Lin and Ding\cite{dsj}, Wang\cite{wang}. We are interesting in the local and global existence of strong solution to system (\ref{prob1})-(\ref{ini1}) in this paper. We will also consider a criterion for system (\ref{prob1}), which is similar to serrin's criterion on regularity of weak solution to Navier-Stokes equations. A natural way of dealing with uniqueness is to find a function space as large as possible, where the existence and uniqueness of solution hold. In a word, we need to find a ''critical'' space. This approach has been initiated by Fujita and Kato \cite{fujta} on Navier-Stokes equations. We give the scaling of $(u,p,d)$ as follows
\begin{eqnarray}
u_\lambda(t,x)=\lambda u(\lambda^2t,\lambda x),\ p_\lambda(t,x)=\lambda^2p(\lambda^2t,\lambda x),\ d_\lambda(t,x)=d(\lambda^2t,\lambda x),\nonumber
\end{eqnarray}
where $(u,p,d)$ is a solution of (\ref{prob1})-(\ref{ini1}). We note that $(u_\lambda,p_\lambda,d_\lambda)$ still satisfies system (\ref{prob1}) with initial data
\begin{eqnarray}
u_{\lambda 0}=\lambda u_0(\lambda x),\ d_{\lambda 0}=d_0(\lambda x).\nonumber
\end{eqnarray}
So it is natural to give the following definition:
a space $X$ is called critical space if
\begin{eqnarray}
\forall\ u\in X,\qquad \|u_\lambda\|_X=\|u\|_X,\quad for\ all\ \lambda>0.\nonumber
\end{eqnarray}
In \cite{fujta}, Fujita and Kato used $\dot{H}^{\frac{1}{2}}$ as the critical space on Navier-Stokes equations in three-dimensional space. In this paper, we consider the homogeneous Besov spaces $\dot{B}_{p,r}^s$. Inspired by \cite{Danchin1}, we find
\begin{eqnarray}
\|u_\lambda(t,\cdot)\|_{\dot{B}_{2,1}^{\frac{\mathbb{N}}{2}-1}}=\|u(t,\cdot)\|_{\dot{B}_{2,1}^{\frac{\mathbb{N}}{2}-1}},\quad \|d_\lambda(t,\cdot)\|_{\dot{B}_{2,1}^{\frac{\mathbb{N}}{2}}}=\|d(t,\cdot)\|_{\dot{B}_{2,1}^{\frac{\mathbb{N}}{2}}},\nonumber
\end{eqnarray}
where $\mathbb{N}$ denotes the spacial dimension. In addition, we suppose $|d|=1$ in physics and note that
\begin{eqnarray}
\dot{H}^{\frac{\mathbb{N}}{2}-1}\nsubseteq L^\infty,\qquad \dot{B}_{2,1}^{\frac{\mathbb{N}}{2}}\hookrightarrow L^\infty.\nonumber
\end{eqnarray}
So it is natural to choose $\dot{B}_{2,1}^{\frac{\mathbb{N}}{2}-1}$(for $u$) and $\dot{B}_{2,1}^{\frac{\mathbb{N}}{2}}$(for $d$) as the critical spaces in which we study liquid crystals system (\ref{prob1}).
We are able to prove the following theorems.
\begin{thm}\label{th1}
Let $\mathbb{N}\geq2$, $\bar{d_0}$ be a constant vector with $|\bar{d_0}|=1$, $u_0\in\dot{B}_{2,1}^{\frac{\mathbb{N}}{2}-1}$, $(d_0-\bar{d_0})\in\dot{B}_{2,1}^{\frac{\mathbb{N}}{2}}$ in (\ref{ini}), (\ref{ini1}), and
\begin{eqnarray}
E_0=\|u_0\|_{\dot{B}_{2,1}^{\frac{\mathbb{N}}{2}-1}}+\|(d_0-\bar{d_0})\|_{\dot{B}_{2,1}^{\frac{\mathbb{N}}{2}}}.\nonumber
\end{eqnarray}
\flushleft
\begin{enumerate}
\begin{description}
\item[I.] There exists a constant $T>0$, such that system (\ref{prob1})-(\ref{ini1}) has a unique strong solution $(u,\ d)$ on $[0,T]\times R^\mathbb{N}$,
\begin{eqnarray}
(u,\ d-\bar{d_0})\in\widetilde{L}^1_T(\dot{B}_{2,1}^{\frac{\mathbb{N}}{2}+1})\cap C([0,T];\dot{B}_{2,1}^{\frac{\mathbb{N}}{2}-1})\times\widetilde{L}^1_T(\dot{B}_{2,1}^{\frac{\mathbb{N}}{2}+2})\cap C([0,T];\dot{B}_{2,1}^{\frac{\mathbb{N}}{2}}).\nonumber
\end{eqnarray}
In addition, for any $\rho_1,\ \rho_2\in [1,\infty]$,
\begin{eqnarray}
\|u\|_{\widetilde{L}^{\rho_1}_T(\dot{B}_{2,1}^{\frac{\mathbb{N}}{2}-1+\frac{2}{\rho_1}})}+
\|(d-\bar{d_0})\|_{\widetilde{L}^{\rho_2}_T(\dot{B}_{2,1}^{\frac{\mathbb{N}}{2}+\frac{2}{\rho_2}})}\leq C(E_0)\nonumber
\end{eqnarray}
holds.
\item[II.] There exists a constant $\delta_0\geq0$, such that if $E_0\leq\delta_0$, system (\ref{prob1})-(\ref{ini1}) has a unique global strong solution $(u,\ d)$ on $[0,\infty)\times R^\mathbb{N}$,
\begin{eqnarray}
(u,\ d-\bar{d_0})\in\widetilde{L}^1(\dot{B}_{2,1}^{\frac{\mathbb{N}}{2}+1})\cap C([0,\infty);\dot{B}_{2,1}^{\frac{\mathbb{N}}{2}-1})\times\widetilde{L}^1(\dot{B}_{2,1}^{\frac{\mathbb{N}}{2}+2})\cap C([0,\infty);\dot{B}_{2,1}^{\frac{\mathbb{N}}{2}}).\nonumber
\end{eqnarray}
In addition, for any $\rho_1,\ \rho_2\in [1,\infty]$,
\begin{eqnarray}
\|u\|_{\widetilde{L}^{\rho_1}(\dot{B}_{2,1}^{\frac{\mathbb{N}}{2}-1+\frac{2}{\rho_1}})}+
\|(d-\bar{d_0})\|_{\widetilde{L}^{\rho_2}(\dot{B}_{2,1}^{\frac{\mathbb{N}}{2}+\frac{2}{\rho_2}})}\leq C(\delta_0)\nonumber
\end{eqnarray}
holds.
\end{description}
\end{enumerate}
\end{thm}
Space $\widetilde{L}^\rho_T(\dot{B}^s_{p,r})$ and $\widetilde{L}^\rho(\dot{B}^s_{p,r})$ will be introduced in section \ref{2}.
We want to point out that if $(u,\ d)$ is smooth, then we can proof $|d|=|d_0|=1$ by maximal principle. Wang\cite{wang}, Lin and Ding\cite{dsj} proved $|d|=|d_0|=1$ by taking $|\nabla d|^2d$ as a second fundamental form.
The following result is inspired by Kozono and Shimada\cite{hk}.
\begin{thm}\label{th2}
Let $(u,\ d)$ be a strong solution of (\ref{prob1})-(\ref{ini1}) on $[0,T]\times R^{\mathbb{N}}$ in Theorem \ref{th1}. For any $T'>T$, such that if
\begin{eqnarray}
&u\in\widetilde{L}_{T'}^{\rho_1}(\dot{B}_{\infty,\infty}^{-1+\frac{2}{\rho_1}}),\  d-\bar{d_0}\in\widetilde{L}_{T'}^{\rho_2}(\dot{B}_{\infty,\infty}^{\frac{2}{\rho_2}})\ and\ d-\bar{d_0}\in\widetilde{L}_{T'}^{\rho_3}(\dot{B}_{2,\infty}^{\frac{\mathbb{N}}{2}+\frac{2}{\rho_3}}),\nonumber\\
&\frac{\mathbb{N}}{2}+\frac{2}{\rho_2}+\frac{2}{\rho_3}-2>0,\nonumber\\
&\|u\|_{\widetilde{L}_{T'}^{\rho_1}(\dot{B}_{\infty,\infty}^{-1+\frac{2}{\rho_1}})}+ \|d-\bar{d_0}\|_{\widetilde{L}_{T'}^{\rho_2}(\dot{B}_{\infty,\infty}^{\frac{2}{\rho_2}})}+
\|d-\bar{d_0}\|_{\widetilde{L}_{T'}^{\rho_3}(\dot{B}_{2,\infty}^{\frac{\mathbb{N}}{2}+\frac{2}{\rho_3}})}<\infty\label{serrin1}
\end{eqnarray}
hold, where $(\rho_1,\ \rho_2,\ \rho_3)\in (2,\infty)^3$,
then $(u,\ d)$ is a strong solution on $[0,T']\times R^{\mathbb{N}}$ in Theorem \ref{th1}.
\end{thm}
It is well-known that using Besov space, one can get the same results in a lower regularity. And the following Remark is very interested.
\begin{remark}
By the embedding theory
\begin{eqnarray}
\widetilde{L}_T^{\rho_1}(\dot{B}_{p,\infty}^{\frac{\mathbb{N}}{p}-1+\frac{2}{\rho_1}})\hookrightarrow \widetilde{L}_T^{\rho_1}(\dot{B}_{\infty,\infty}^{-1+\frac{2}{\rho_1}}),\quad L^p\hookrightarrow \dot{B}_{p,\infty}^{0}\hookrightarrow \dot{B}_{\infty,\infty}^{-\frac{\mathbb{N}}{p}},\nonumber
\end{eqnarray}
we have
\begin{eqnarray}
L^{\rho_1}_TL^p\hookrightarrow \widetilde{L}_T^{\rho_1}(\dot{B}_{\infty,\infty}^{-1+\frac{2}{\rho_1}}),\quad -1+\frac{2}{\rho_1}=-\frac{\mathbb{N}}{p},\nonumber
\end{eqnarray}
which is the Serrin's criterion on endpoint.
\end{remark}
\begin{remark}
The proof of Theorem \ref{th2} implies that the life-span can be extended a little larger under the condition (\ref{serrin1}). Also we obtain that if $[0,T)$ is the life-span of strong solution to (\ref{prob1})-(\ref{ini1}), then (\ref{serrin1}) fails, that is
\begin{eqnarray}
\lim_{a\rightarrow T}\left(\|u\|_{\widetilde{L}_a^{\rho_1}(\dot{B}_{\infty,\infty}^{-1+\frac{2}{\rho_1}})}+ \|d-\bar{d_0}\|_{\widetilde{L}_a^{\rho_2}(\dot{B}_{\infty,\infty}^{\frac{2}{\rho_2}})}+
\|d-\bar{d_0}\|_{\widetilde{L}_a^{\rho_3}(\dot{B}_{2,\infty}^{\frac{\mathbb{N}}{2}+\frac{2}{\rho_3}})}\right)=\infty.\nonumber
\end{eqnarray}
\end{remark}
  The rest of our paper is organized as follows. In section \ref{2}, we give a short introduction on Besov space. In section \ref{3}, we prove Theorem \ref{th1} and Theorem \ref{th2}.

\section{Littlewood-Paley theory and Besov space}\label{2}
In this section, we present some well-known facts on Littlewood-paley theory, more details see \cite{book2,book1}. Let $S(R^\mathbb{N})$ be the Schwartz space and $S'(R^\mathbb{N})$ be its dual space. Let $(\chi,\varphi)$ be a couple of smooth functions such that
\begin{eqnarray}
&&\chi(x)=\left\{\begin{array}{ll}
1& x\in \mathcal{B}(0,\frac{3}{4}), \\
0&x\in R^\mathbb{N}\setminus \mathcal{B}(0,\frac{4}{3}),
\end{array}\right.\nonumber\\
&&\varphi(x)=\chi\left(\frac{x}{2}\right)-\chi(x).\nonumber
\end{eqnarray}
Then we have $\mbox{{\bf supp}}\varphi\subset \mathcal{C}(0,\frac{3}{4},\frac{8}{3})$, $\varphi\geq0,\ x\in R^\mathbb{N}$. Here we denote $\mathcal{B}(0,R)$ as an open ball with radius $R$ centered at zero, and $\mathcal{C}(0,R_1,R_2)$ as an annulus $\{x\in R^\mathbb{N}\ |\ R_1\leq|x|\leq R_2\}$. Then
 \begin{eqnarray}
\varphi_q(x):=\varphi\left(\frac{x}{2^q}\right),\ \mathrm{and}\ \mbox{{\bf supp}}\varphi_q\in \mathcal{C}\left(0,2^q\frac{3}{4},2^q\frac{8}{3}\right),\nonumber
\end{eqnarray}
For any $u\in S(R^\mathbb{N})$, the Fourier transform of $u$ denotes by $\widehat{u}$ or $\mathscr{F}u$. The inverse Fourier transform denotes by $\mathscr{F}^{-1}$. Let $h=\mathscr{F}^{-1}\varphi$. We define homogeneous dyadic blocks as follows
\begin{eqnarray}
&&\bigtriangleup_qu:=\varphi(2^{-q}D)u=2^{q\mathbb{N}}\int_{\mathbb{N}}h(2^qy)u(x-y)dy,\nonumber\\
&&S_qu:=\sum_{p\leq q-1}\bigtriangleup_pu.\nonumber
\end{eqnarray}
Then for any $u\in S'(R^\mathbb{N})$, the following decomposition
\begin{eqnarray}
u=\sum_{q=-\infty}^\infty\bigtriangleup_qu\nonumber
\end{eqnarray}
is called homogeneous Littlewood-Paley decomposition.\\
Let $u\in S'(R^\mathbb{N})$, $s$ be a real number, and $1\leq p,r\leq\infty$. We set
\begin{eqnarray}
\|u\|_{\dot{B}^s_{p,r}}:=\left(\sum_{q=-\infty}^\infty(2^{qs}\|\bigtriangleup_qu\|_{L^p})^r\right)^{\frac{1}{r}}.\nonumber
\end{eqnarray}
The besov space $\dot{B}^s_{p,r}$ is defined as follows:
\begin{eqnarray}
\dot{B}^s_{p,r}:=\left\{u\in S'(R^\mathbb{N})\left|\|u\|_{\dot{B}^s_{p,r}}\right.<\infty\right\}.\nonumber
\end{eqnarray}
\begin{dif}
Let $u\in S'((0,T)\times R^\mathbb{N})$, $s\in R$ and $1\leq p,r,\rho\leq\infty$. We define
\begin{eqnarray}
\|u\|_{\widetilde{L}^\rho_T(\dot{B}^s_{p,r})}:=\left(\sum_{q=-\infty}^\infty(2^{qs}\|\bigtriangleup_qu\|_{L^\rho_TL^p})^r\right)^{\frac{1}{r}}\nonumber
\end{eqnarray}
and
\begin{eqnarray}
\widetilde{L}^\rho_T(\dot{B}^s_{p,r}):=\left\{u\in S'((0,T)\times R^\mathbb{N})\left|\|u\|_{\widetilde{L}^\rho_T(\dot{B}^s_{p,r})}\right.<\infty\right\}.\nonumber
\end{eqnarray}
\end{dif}
Similarly, we set
\begin{eqnarray}
&&\|u\|_{\widetilde{L}^\rho(\dot{B}^s_{p,r})}:=\left(\sum_{q=-\infty}^\infty(2^{qs}\|\bigtriangleup_qu\|_{L^\rho(R^+;L^p)})^r\right)^{\frac{1}{r}},\nonumber\\
&&\|u\|_{\widetilde{L}_{(T_1,T_2)}^\rho(\dot{B}^s_{p,r})}:=\left(\sum_{q=-\infty}^\infty(2^{qs}\|\bigtriangleup_qu\|_{L^\rho\left((T_1,T_2);L^p\right)})^r\right)^{\frac{1}{r}},
\end{eqnarray}
and define the corresponding spaces $\widetilde{L}^\rho(\dot{B}^s_{p,r}),\ \widetilde{L}_{(T_1,T_2)}^\rho(\dot{B}^s_{p,r})$ respectively.  $L_T^\rho(\dot{B}^s_{p,r})$ is defined as usual.
\begin{remark}
According to Minkowski inequality, we have
\begin{eqnarray}
\|u\|_{\widetilde{L}_T^\rho(\dot{B}^s_{p,r})}\leq\|u\|_{L_T^\rho(\dot{B}^s_{p,r})}\quad if\quad r\geq\rho,\nonumber\\
\|u\|_{\widetilde{L}_T^\rho(\dot{B}^s_{p,r})}\geq\|u\|_{L_T^\rho(\dot{B}^s_{p,r})}\quad if\quad r\leq\rho.\nonumber
\end{eqnarray}
\end{remark}
Next let's present some results of paradifferential calculus in homogeneous space, which is useful for dealing with nonlinear equations. For $u,\ v\in S'(R^{\mathbb{N}})$, we have the following formal decomposition:
\begin{eqnarray}
uv=\sum_{p,q}\bigtriangleup_pu\bigtriangleup_qv.\nonumber
\end{eqnarray}
Define the operators $\mathcal{T},\ \mathcal{R}$ as follows,
\begin{eqnarray}
&&\mathcal{T}_uv=\sum_{p\leq q-2}\bigtriangleup_pu\bigtriangleup_qv=\sum_{q}S_{q-1}u\bigtriangleup_qv,\nonumber\\
&&\mathcal{R}(u,v)=\sum_{|p-q|\leq1}\bigtriangleup_pu\bigtriangleup_qv.\nonumber
\end{eqnarray}
Then we have the following so-called homogeneous Bony decomposition:
\begin{eqnarray}
uv=\mathcal{T}_uv+\mathcal{T}_vu+\mathcal{R}(u,v).\nonumber
\end{eqnarray}
In order to estimate the above terms, let's recall some lemmas (see \cite{book2,book1}).
\begin{lema}\label{lemma1}
There exists a constant $C$, such that, for any couple of real numbers $(s,\sigma)$, with $\sigma$ positive and any $(p,\ r,\ r_1,\ r_2)$ in $[1,\infty]^4$ with $\frac{1}{r}=\frac{1}{r_1}+\frac{1}{r_2}$, we have
\begin{eqnarray}
&\|\mathcal{T}\|_{\mathcal{L}(L^\infty\times\dot{B}^s_{p,r};\dot{B}^s_{p,r})}\leq C\quad if\ s<\frac{\mathbb{N}}{p}\ or\ s=\frac{\mathbb{N}}{p},\ r=1,\nonumber\\
&\|\mathcal{T}\|_{\mathcal{L}(\dot{B}^{-\sigma}_{\infty,r_1}\times\dot{B}^s_{p,r_2};\dot{B}^{s-\sigma}_{p,r})}\leq C\quad if\ s-\sigma<\frac{\mathbb{N}}{p}\ or\ s-\sigma=\frac{\mathbb{N}}{p},\ r=1.\nonumber
\end{eqnarray}
\end{lema}
\begin{lema}\label{lemma2}
There exists a constant $C$, such that, for any real numbers $s_1,s_2$, and $(p,\ p_1,\ p_2,\ r,\ r_1,\ r_2)$ in $[1,\infty]^6$ with
\begin{eqnarray}
s_1+s_2>0,\ \frac{1}{p}\leq\frac{1}{p_1}+\frac{1}{p_2}\leq1,\ \frac{1}{r}\leq\frac{1}{r_1}+\frac{1}{r_2}\leq1,\nonumber
\end{eqnarray}
we have
\begin{eqnarray}
\|\mathcal{R}\|_{\mathcal{L}(\dot{B}^{s_1}_{p_1,r_1}\times\dot{B}^{s_2}_{p_2,r_2};\dot{B}^{\sigma}_{p,r})}\leq C,\quad \sigma=s_1+s_2-\mathbb{N}(\frac{1}{p_1}+\frac{1}{p_2}-\frac{1}{p}),\nonumber
\end{eqnarray}
providing that $\sigma<\frac{\mathbb{N}}{p}$, or $\sigma=\frac{\mathbb{N}}{p}$ and $r=1$.
\end{lema}
The initial problem of heat equation reads
\begin{eqnarray}
\left\{\begin{array}{ll}
& v_t-a\Delta v=G, \\
& v|_{t=0}=v_0.
\end{array}\right.\label{heat}
\end{eqnarray}
The following lemma can be found in \cite{book2,book1}.
\begin{lema}\label{lemma3}
Let $T>0,\ s\in R$, and $[\rho,\ p,\ r]\in [1,\infty]^3$. Assume that $v_0\in\dot{B}^s_{p,r}$ and $G\in\widetilde{L}^\rho(\dot{B}^{s-2+\frac{2}{\rho}}_{p,r})$, Then (\ref{heat}) has a unique solution $v\in \widetilde{L}^1(\dot{B}^{s}_{p,r})\cap\widetilde{L}^\infty(\dot{B}^{s-2}_{p,r})$, satisfying
\begin{eqnarray}
a^{\frac{1}{\rho_1}}\|v\|_{\widetilde{L}^\rho(\dot{B}^{s+\frac{2}{\rho}}_{p,r})}\leq C\left(\|v_0\|_{\dot{B}^s_{p,r}}+a^{\frac{1}{\rho}-1}\|G\|_{\widetilde{L}^\rho(\dot{B}^{s-2+\frac{2}{\rho}}_{p,r})}\right),\quad \rho_1\in[\rho,\infty],\nonumber
\end{eqnarray}
in addition if $r$ is finite, then $v\in C([0,T];\dot{B}^s_{p,r})$.
\end{lema}

\section{Proof of main results}\label{3}
Let $\mathcal{P}$ denote the Leray projector on solenoidal vector fields which is defined by
\begin{eqnarray}
\mathcal{P}=I-\nabla\Delta^{-1}\mbox{{\bf div}}.\nonumber
\end{eqnarray}
By operator $\mathcal{P}$, we can project the second equation of (\ref{prob1}) onto the divergence free vector field. Then the pressure $p$ can be eliminated.
It is easy to notice that $\mathcal{P}$ is a homogeneous multiplier of degree zero. Denote $\tau=d-\bar{d_0}$. We only need to consider the following equations
\begin{eqnarray}
\left\{\begin{array}{ll}
&u_t-\mu\Delta u+\mathcal{P}\left[u\cdot\nabla u+\nabla(\nabla \tau\odot\nabla \tau)\right]=0,\\
&\tau_t-\Delta \tau+ u\cdot\nabla \tau-|\nabla \tau|^2\tau-|\nabla \tau|^2\bar{d_0}=0,
\end{array}\right.\label{prob1c}
\end{eqnarray}
with initial conditions
\begin{eqnarray}
u|_{t=0}=u_0(x),\qquad \tau|_{t=0}=\tau_0=d_0(x)-\bar{d_0},\label{inic}
\end{eqnarray}
and far field behaviors
\begin{eqnarray}
u\rightarrow 0,\ \tau\rightarrow 0,\ as\ |x|\rightarrow\infty.\label{inic1}
\end{eqnarray}
Instead of proving  Theorem \ref{th1} and Theorem \ref{th2}, we should only consider the following theorems.
\begin{thm}\label{th3}
Let $\mathbb{N}\geq2$, $u_0\in\dot{B}_{2,1}^{\frac{\mathbb{N}}{2}-1}$, $\tau_0\in\dot{B}_{2,1}^{\frac{\mathbb{N}}{2}}$ in (\ref{inic}),
and
\begin{eqnarray}
E_0=\|u_0\|_{\dot{B}_{2,1}^{\frac{\mathbb{N}}{2}-1}}+\|\tau_0\|_{\dot{B}_{2,1}^{\frac{\mathbb{N}}{2}}}.\nonumber
\end{eqnarray}
\flushleft
\begin{enumerate}
\begin{description}
\item[I.] There exists a constant $T>0$, such that system (\ref{prob1c})-(\ref{inic1}) has a unique strong solution
 $(u,\tau)$ on $[0,T]\times R^\mathbb{N}$,
\begin{eqnarray}
(u,\ \tau)\in\widetilde{L}^1_T(\dot{B}_{2,1}^{\frac{\mathbb{N}}{2}+1})\cap C([0,T];\dot{B}_{2,1}^{\frac{\mathbb{N}}{2}-1})\times\widetilde{L}^1_T(\dot{B}_{2,1}^{\frac{\mathbb{N}}{2}+2})\cap C([0,T];\dot{B}_{2,1}^{\frac{\mathbb{N}}{2}}).\nonumber
\end{eqnarray}
In addition, for any $\rho_1,\ \rho_2\in [1,\infty]$,
\begin{eqnarray}
\|u\|_{\widetilde{L}^{\rho_1}_T(\dot{B}_{2,1}^{\frac{\mathbb{N}}{2}-1+\frac{2}{\rho_1}})}+
\|\tau\|_{\widetilde{L}^{\rho_2}_T(\dot{B}_{2,1}^{\frac{\mathbb{N}}{2}+\frac{2}{\rho_2}})}\leq C(E_0)\label{e1}
\end{eqnarray}
holds.
\item[II.] There exists a constant $\delta_0\geq0$, such that if $E_0\leq\delta_0$, system (\ref{prob1c})-(\ref{inic1}) has a unique global strong solution $(u,\ \tau)$ on $[0,\infty)\times R^\mathbb{N}$,
\begin{eqnarray}
(u,\ \tau)\in\widetilde{L}^1(\dot{B}_{2,1}^{\frac{\mathbb{N}}{2}+1})\cap C([0,\infty);\dot{B}_{2,1}^{\frac{\mathbb{N}}{2}-1})\times\widetilde{L}^1(\dot{B}_{2,1}^{\frac{\mathbb{N}}{2}+2})\cap C([0,\infty);\dot{B}_{2,1}^{\frac{\mathbb{N}}{2}}).\nonumber
\end{eqnarray}
In addition, for any $\rho_1,\ \rho_2\in [1,\infty]$,
 \begin{eqnarray}
\|u\|_{\widetilde{L}^{\rho_1}(\dot{B}_{2,1}^{\frac{\mathbb{N}}{2}-1+\frac{2}{\rho_1}})}+
\|\tau\|_{\widetilde{L}^{\rho_2}(\dot{B}_{2,1}^{\frac{\mathbb{N}}{2}+\frac{2}{\rho_2}})}\leq C(\delta_0)\label{e2}
\end{eqnarray}
holds.
\end{description}
\end{enumerate}
\end{thm}
\begin{thm}\label{th4}
Let $(u,\ \tau)$ be a strong solution of (\ref{prob1c})-(\ref{inic1}) on $[0,T]\times R^{\mathbb{N}}$ in Theorem \ref{th3}. For any $T'>T$, such that if
\begin{eqnarray}
&u\in\widetilde{L}_{T'}^{\rho_1}(\dot{B}_{\infty,\infty}^{-1+\frac{2}{\rho_1}}),\  \tau\in\widetilde{L}_{T'}^{\rho_2}(\dot{B}_{\infty,\infty}^{\frac{2}{\rho_2}})\ and\ \tau\in\widetilde{L}_{T'}^{\rho_3}(\dot{B}_{2,\infty}^{\frac{\mathbb{N}}{2}+\frac{2}{\rho_3}}),\nonumber\\
&\frac{\mathbb{N}}{2}+\frac{2}{\rho_2}+\frac{2}{\rho_3}-2>0,\nonumber\\
&\|u\|_{\widetilde{L}_{T'}^{\rho_1}(\dot{B}_{\infty,\infty}^{-1+\frac{2}{\rho_1}})}+ \|\tau\|_{\widetilde{L}_{T'}^{\rho_2}(\dot{B}_{\infty,\infty}^{\frac{2}{\rho_2}})}+
\|\tau\|_{\widetilde{L}_{T'}^{\rho_3}(\dot{B}_{2,\infty}^{\frac{\mathbb{N}}{2}+\frac{2}{\rho_3}})}<\infty\label{serrin}
\end{eqnarray}
hold, where $(\rho_1,\ \rho_2,\ \rho_3)\in (2,\infty)^3$,
then $(u,\ \tau)$ is a strong solution on $[0,T']\times R^{\mathbb{N}}$ in Theorem \ref{th3}.
\end{thm}
Next let us introduce the heat semi-group operator $e^{at\Delta}$. Let
\begin{eqnarray}
v=e^{at\Delta}v_0,\nonumber
\end{eqnarray}
then $v$ solves problem (\ref{heat}) with $G$ replaced by zero.
\begin{lema}\label{lemma4}
Let $v=e^{at\Delta}v_0$, $v_0\in \dot{B}^s_{2,1},\ s\in R$.
\flushleft
\begin{enumerate}
\begin{description}
\item[I.] Assume $\|v_0\|_{\dot{B}^s_{2,1}}\leq C_0$, $\rho\in [1,\infty)$. For any small $\varepsilon_0>0$, there exists $T_0$, such that the following estimate holds
\begin{eqnarray}
\|v\|_{\widetilde{L}_T^\rho(\dot{B}^{s+\frac{2}{\rho}}_{2,1})}\leq \varepsilon_0,\ \ \mathrm{for\ any}\ T\leq T_0.\label{l1}
\end{eqnarray}
\item[II.] Assume $\rho\in [1,\infty]$. For any small $\varepsilon_0>0$, there exists $\delta_0$, such that if $\|v_0\|_{\dot{B}^s_{2,1}}\leq \delta_0$, the following estimate holds
\begin{eqnarray}
\|v\|_{\widetilde{L}^\rho(\dot{B}^{s+\frac{2}{\rho}}_{2,1})}\leq \varepsilon_0.\label{l2}
\end{eqnarray}
\end{description}
\end{enumerate}
\end{lema}
\begin{proof}
Firstly, we know that
\begin{eqnarray}
\widehat{v}(t,\xi)=e^{-at|\xi|^2}\widehat{v}_0(\xi).\nonumber
\end{eqnarray}
Then we have
\begin{eqnarray}
\bigtriangleup_qv(t)=\mathscr{F}^{-1}(e^{-at|\xi|^2}\varphi_q(\xi)\widehat{v}_0(\xi))=e^{at\Delta}\bigtriangleup_qv_0,\nonumber
\end{eqnarray}
and
\begin{eqnarray}
\|\bigtriangleup_qv(t)\|_{L^2}\leq e^{-ka2^{2q}t}\|\bigtriangleup_qv_0\|_{L^2},\label{hini}
\end{eqnarray}
where $k$ is a constant.\\
\noindent{\bf Case I.}  By using (\ref{hini}), we obtain
\begin{eqnarray}
\|\bigtriangleup_qv(t)\|_{L^\rho_TL^2}\leq \left(\frac{1-e^{-kaT\rho2^{2q}}}{ka\rho2^{2q}}\right)^{\frac{1}{\rho}}\|\bigtriangleup_qv_0\|_{L^2}.\nonumber
\end{eqnarray}
It is easy to find $N>0$, such that
\begin{eqnarray}
\sum_{|q|> N}\left(\frac{1-e^{-kaT\rho2^{2q}}}{ka\rho2^{2q}}\right)^{\frac{1}{\rho}}2^{qs}\|\bigtriangleup_qv_0\|_{L^2}\leq \frac{\varepsilon_0}{2}.\label{l3}
\end{eqnarray}
For above fixed $N$, let $T$ be small enough such that
\begin{eqnarray}
\sum_{|q|\leq N}\left(\frac{1-e^{-kaT\rho2^{2q}}}{ka\rho2^{2q}}\right)^{\frac{1}{\rho}}2^{qs}\|\bigtriangleup_qv_0\|_{L^2}\leq \frac{\varepsilon_0}{2}.\label{l4}
\end{eqnarray}
Combining (\ref{l3}) and (\ref{l4}), we obtain (\ref{l1}).\\
\noindent{\bf Case II.}
By using (\ref{hini}), we have
\begin{eqnarray}
\|\bigtriangleup_qv(t)\|_{L^\rho L^2}\leq \left(\frac{1}{ka\rho2^{2q}}\right)^{\frac{1}{\rho}}\|\bigtriangleup_qv_0\|_{L^2}.\nonumber
\end{eqnarray}
Then
\begin{eqnarray}
\|v\|_{L^\rho(\dot{B}^{s+\frac{2}{\rho}}_{2,1})}\leq C\|v_0\|_{\dot{B}^{s}_{2,1}}.\nonumber
\end{eqnarray}
By choosing $\delta_0\leq \frac{\varepsilon_0}{C}$, we finish the proof.
\end{proof}
\noindent{\bf Proof of Theorem \ref{th3}.}\\
\par In this part, we will use the iterative method to prove Theorem  \ref{th3}, which is separated into three steps. Firstly we point out that the case I in the following proof corresponds to I in Theorem  \ref{th3} and case II corresponds to II in Theorem  \ref{th3}. Also we should keep in mind that $\dot{B}^{s}_{2,1},\ s>\frac{\mathbb{N}}{2}$ is not a Banach space. Then let us consider the following linear equations,
\begin{eqnarray}
\left\{\begin{array}{ll}
&\partial_tu_n-\mu\Delta u_n=-\mathcal{P}\left[u_{n-1}\cdot\nabla u_{n-1}+\nabla(\nabla \tau_{n-1}\odot\nabla \tau_{n-1})\right],\\
&\partial_t\tau_n-\Delta \tau_n=- u_{n-1}\cdot\nabla \tau_{n-1}+|\nabla \tau_{n-1}|^2\tau_{n-1}+|\nabla \tau_{n-1}|^2\bar{d_0},
\end{array}\right.\label{linear}
\end{eqnarray}
with initial conditions
\begin{eqnarray}
&u_n|_{t=0}=u_0,\nonumber\\
&\tau_n|_{t=0}=\tau_0.\nonumber
\end{eqnarray}
Let us set $u_1=e^{at\Delta}u_0$, $\tau_1=e^{at\Delta}\tau_0$ and begin our proof.\\
\noindent {\bf First step: Uniform boundedness}\\
\noindent{\bf Case I.}
We claim that the following estimates hold for some $T>0$,
\begin{eqnarray}
&&\|u_n\|_{\widetilde{L}^4_T(\dot{B}_{2,1}^{\frac{\mathbb{N}}{2}-\frac{1}{2}})}+
\|\tau_n\|_{\widetilde{L}^4_T(\dot{B}_{2,1}^{\frac{\mathbb{N}}{2}+\frac{1}{2}})}\leq (C+1)\varepsilon_0,\quad n=1,2,3\cdots,\label{unb1}\\
&&\|\tau_n\|_{\widetilde{L}^\infty_T(\dot{B}_{2,1}^{\frac{\mathbb{N}}{2}})}\leq(C+1)E_0,\quad n=1,2,3\cdots.\label{unb2}
\end{eqnarray}
Here $C$ is an absolute constant.
Indeed by Lemma \ref{lemma4}, we can find (\ref{unb1}) and (\ref{unb2}) hold for $n=1$, providing $T$ is small.
Assume (\ref{unb1}) and (\ref{unb2}) hold for $n-1$. By using Lemma \ref{lemma1}-Lemma \ref{lemma3}, we can choose $\varepsilon_0\leq\frac{1}{(C+1)\sqrt{\left(C+1+\frac{2}{E_0}\right)}}$ to obtain
\begin{align}
\|\tau_n\|_{\widetilde{L}^\infty_T(\dot{B}_{2,1}^{\frac{\mathbb{N}}{2}})}&\leq CE_0+(C+1)^2\varepsilon^2_0(C+1)E_0+2(C+1)^2\varepsilon^2_0\nonumber\\
&\leq E_0\left(C+(C+1)^3\varepsilon^2_0+\frac{2(C+1)^2\varepsilon^2_0}{E_0}\right)\nonumber\\
&\leq (C+1)E_0,\nonumber
\end{align}
where $C$ is the constant appeared in Lemma \ref{lemma1}-Lemma \ref{lemma3}.
 Then by Lemma \ref{lemma1}-Lemma \ref{lemma3} again, we have
\begin{align}
\|u_n\|_{\widetilde{L}^4_T(\dot{B}_{2,1}^{\frac{\mathbb{N}}{2}-\frac{1}{2}})}+
\|\tau_n\|_{\widetilde{L}^4_T(\dot{B}_{2,1}^{\frac{\mathbb{N}}{2}+\frac{1}{2}})}&\leq C\varepsilon_0+4(C+1)^2\varepsilon^2_0+(C+1)^3\varepsilon^2_0E_0\nonumber\\
&\leq(C+1)\varepsilon_0,\nonumber
\end{align}
by choosing $\varepsilon_0\leq \min\{\frac{1}{(C+1)\sqrt{\left(C+1+\frac{2}{E_0}\right)}},\ \frac{1}{4(C+1)\left[1+(C+1)E_0\right]}\}$. So (\ref{unb1}) and (\ref{unb2}) are proved.
In addition, we can get from (\ref{unb1}), (\ref{unb2}) and Lemma \ref{lemma3} that
\begin{eqnarray}
\|u_n\|_{\widetilde{L}^\infty_T(\dot{B}_{2,1}^{\frac{\mathbb{N}}{2}-1})\cap\widetilde{L}^1_T(\dot{B}_{2,1}^{\frac{\mathbb{N}}{2}+1})}+
\|\tau_n\|_{\widetilde{L}^\infty_T(\dot{B}_{2,1}^{\frac{\mathbb{N}}{2}})\cap\widetilde{L}^1_T(\dot{B}_{2,1}^{\frac{\mathbb{N}}{2}+2})}\leq C(E_0),\nonumber
\end{eqnarray}
and more precisely,
\begin{eqnarray}
\|u_n\|_{\widetilde{L}^\rho_T(\dot{B}_{2,1}^{\frac{\mathbb{N}}{2}-1+\frac{2}{\rho}})}+
\|\tau_n\|_{\widetilde{L}^\rho_T(\dot{B}_{2,1}^{\frac{\mathbb{N}}{2}+\frac{2}{\rho}})}\leq (C(\rho)+1)\varepsilon_0,\quad \rho\in[1,\infty).\label{small}
\end{eqnarray}
\noindent{\bf Case II.}
For the small initial data, it is simple to prove
\begin{eqnarray}
\|u_n\|_{\widetilde{L}^\rho(\dot{B}_{2,1}^{\frac{\mathbb{N}}{2}-1+\frac{2}{\rho}})}+
\|\tau_n\|_{\widetilde{L}^\rho(\dot{B}_{2,1}^{\frac{\mathbb{N}}{2}+\frac{2}{\rho}})}\leq (C+1)\varepsilon_0,\quad \rho\in[1,\infty],\nonumber
\end{eqnarray}
providing $\delta_0$ is small enough. Here we omit the details.\\
\noindent {\bf Second step: Convergence}\\
\noindent{\bf Case I.} We will prove $\{(u_n,\ \tau_n)|n=1,2\cdots\}$ is a Cauchy sequence. Firstly, let's consider
\begin{eqnarray}
\|u_{m+n+1}-u_{n+1}\|_{\widetilde{L}^\infty_T(\dot{B}_{2,1}^{\frac{\mathbb{N}}{2}-1})}\ and\ \|\tau_{m+n+1}-\tau_{n+1}\|_{\widetilde{L}^\infty_T(\dot{B}_{2,1}^{\frac{\mathbb{N}}{2}})}.\nonumber
\end{eqnarray}
According to the proof of Lemma \ref{lemma3}, we need to estimate the following terms:
\begin{align}
&I_1(t,x)=\int^t_0e^{-\mu(t-s)\Delta}\mathcal{P}\left[\left(u_{m+n}\cdot\nabla u_{m+n}\right)-\left(u_{n}\cdot\nabla u_{n}\right)\right],\nonumber\\
&I_2(t,x)=\int^t_0e^{-\mu(t-s)\Delta}\mathcal{P}\left[\left(\nabla(\nabla \tau_{m+n}\odot\nabla \tau_{m+n})\right)-\left(\nabla(\nabla \tau_{n}\odot\nabla \tau_{n})\right)\right],\nonumber\\
&I_3(t,x)=\int^t_0e^{-(t-s)\Delta}\left[\left(u_{m+n}\cdot\nabla \tau_{m+n}\right)-\left(u_{n}\cdot\nabla \tau_{n}\right)\right],\nonumber\\
&I_4(t,x)=\int^t_0e^{-(t-s)\Delta}\left[\left(|\nabla \tau_{m+n}|^2\tau_{m+n}\right)-\left(|\nabla \tau_{n}|^2\tau_{n}\right)\right],\nonumber\\
&I_5(t,x)=\int^t_0e^{-(t-s)\Delta}\left[\left(|\nabla \tau_{m+n}|^2\bar{d_0}\right)-\left(|\nabla \tau_{n}|^2\bar{d_0}\right)\right].\nonumber
\end{align}
By using (\ref{small}), Lemma \ref{lemma1} and Lemma \ref{lemma2}, we can obtain
\begin{align}
\|I_1(t,x)\|_{\widetilde{L}^2_T(\dot{B}_{2,1}^{\frac{\mathbb{N}}{2}-2})}&\leq
C\|u_{m+n}-u_n\|_{\widetilde{L}^\infty_T(\dot{B}_{2,1}^{\frac{\mathbb{N}}{2}-1})}(\|u_{m+n}\|_{\widetilde{L}^2_T(\dot{B}_{2,1}^{\frac{\mathbb{N}}{2}})}
+\|u_{n}\|_{\widetilde{L}^2_T(\dot{B}_{2,1}^{\frac{\mathbb{N}}{2}})}),\nonumber\\
\|I_2(t,x)\|_{\widetilde{L}^2_T(\dot{B}_{2,1}^{\frac{\mathbb{N}}{2}-2})}&\leq C\|\tau_{m+n}-\tau_n\|_{\widetilde{L}^\infty_T(\dot{B}_{2,1}^{\frac{\mathbb{N}}{2}})}(\|\tau_{m+n}\|_{\widetilde{L}^2_T(\dot{B}_{2,1}^{\frac{\mathbb{N}}{2}+1})}
+\|\tau_{n}\|_{\widetilde{L}^2_T(\dot{B}_{2,1}^{\frac{\mathbb{N}}{2}+1})}),\nonumber\\
\|I_3(t,x)\|_{\widetilde{L}^2_T(\dot{B}_{2,1}^{\frac{\mathbb{N}}{2}-1})}&\leq C\|u_{m+n}-u_n\|_{\widetilde{L}^\infty_T(\dot{B}_{2,1}^{\frac{\mathbb{N}}{2}-1})}\|\tau_{m+n}\|_{\widetilde{L}^2_T(\dot{B}_{2,1}^{\frac{\mathbb{N}}{2}+1})}\nonumber\\
&\quad+C\|u_{n}\|_{\widetilde{L}^2_T(\dot{B}_{2,1}^{\frac{\mathbb{N}}{2}})}\|\tau_{m+n}-\tau_n\|_{\widetilde{L}^\infty_T(\dot{B}_{2,1}^{\frac{\mathbb{N}}{2}})},\nonumber\\
\|I_4(t,x)\|_{\widetilde{L}^2_T(\dot{B}_{2,1}^{\frac{\mathbb{N}}{2}-1})}&\leq C\|\tau_{m+n}-\tau_n\|_{\widetilde{L}^\infty_T(\dot{B}_{2,1}^{\frac{\mathbb{N}}{2}})}\|\tau_{m+n}+\tau_n\|_{\widetilde{L}^2_T(\dot{B}_{2,1}^{\frac{\mathbb{N}}{2}+1})}
\|\tau_{m+n}\|_{\widetilde{L}^\infty_T(\dot{B}_{2,1}^{\frac{\mathbb{N}}{2}})}\nonumber\\
&\quad+C\|\tau_n\|^2_{\widetilde{L}^4_T(\dot{B}_{2,1}^{\frac{\mathbb{N}}{2}+\frac{1}{2}})}\|\tau_{m+n}-\tau_n\|_{\widetilde{L}^\infty_T(\dot{B}_{2,1}^{\frac{\mathbb{N}}{2}})},\nonumber\\
\|I_5(t,x)\|_{\widetilde{L}^2_T(\dot{B}_{2,1}^{\frac{\mathbb{N}}{2}-1})}&\leq C\|\tau_{m+n}-\tau_n\|_{\widetilde{L}^\infty_T(\dot{B}_{2,1}^{\frac{\mathbb{N}}{2}})}\|\tau_{m+n}+\tau_n\|_{\widetilde{L}^2_T(\dot{B}_{2,1}^{\frac{\mathbb{N}}{2}+1})}.\nonumber
\end{align}
By Lemma \ref{lemma3}, we obtain
\begin{eqnarray}
\|u_{m+n+1}-u_{n+1}\|_{\widetilde{L}^\infty_T(\dot{B}_{2,1}^{\frac{\mathbb{N}}{2}-1})}+ \|\tau_{m+n+1}-\tau_{n+1}\|_{\widetilde{L}^\infty_T(\dot{B}_{2,1}^{\frac{\mathbb{N}}{2}})}\nonumber\\
\leq \frac{1}{2}\left(\|u_{m+n}-u_{n}\|_{\widetilde{L}^\infty_T(\dot{B}_{2,1}^{\frac{\mathbb{N}}{2}-1})}+ \|\tau_{m+n}-\tau_{n}\|_{\widetilde{L}^\infty_T(\dot{B}_{2,1}^{\frac{\mathbb{N}}{2}})}\right),\nonumber
\end{eqnarray}
by choosing $\varepsilon_0$ small enough. Then $\{(u_m,\ \tau_m)|n=1,2\cdots\}$ is a Cauchy sequence in Banach space $\widetilde{L}^\infty_T(\dot{B}_{2,1}^{\frac{\mathbb{N}}{2}-1})\times\widetilde{L}^\infty_T(\dot{B}_{2,1}^{\frac{\mathbb{N}}{2}})$.
Let
\begin{eqnarray}
u_m\rightarrow u\quad in\ \widetilde{L}^\infty_T(\dot{B}_{2,1}^{\frac{\mathbb{N}}{2}-1}),\nonumber\\
\tau_m\rightarrow\tau\quad in\ \widetilde{L}^\infty_T(\dot{B}_{2,1}^{\frac{\mathbb{N}}{2}}).\nonumber
\end{eqnarray}
Then $(u,\ \tau)$ is a solution of (\ref{prob1c})-(\ref{inic1}) on $(0,T)\times R^{\mathbb{N}}$ satisfying (\ref{e1}). By Lemma \ref{lemma3}, we can get $(u,\ \tau)\in C([0,T];\dot{B}_{2,1}^{\frac{\mathbb{N}}{2}-1})\times C([0,T];\dot{B}_{2,1}^{\frac{\mathbb{N}}{2}})$.\\
\noindent{\bf Case II.} For the small initial data, we can obtain that
\begin{eqnarray}
\|u_{m+n+1}-u_{n+1}\|_{\widetilde{L}^\infty(\dot{B}_{2,1}^{\frac{\mathbb{N}}{2}-1})}+ \|\tau_{m+n+1}-\tau_{n+1}\|_{\widetilde{L}^\infty(\dot{B}_{2,1}^{\frac{\mathbb{N}}{2}})}\nonumber\\
\leq \frac{1}{2}\left(\|u_{m+n}-u_{n}\|_{\widetilde{L}^\infty(\dot{B}_{2,1}^{\frac{\mathbb{N}}{2}-1})}+ \|\tau_{m+n}-\tau_{n}\|_{\widetilde{L}^\infty(\dot{B}_{2,1}^{\frac{\mathbb{N}}{2}})}\right),\nonumber
\end{eqnarray}
by choosing $\delta_0$ small enough.
The procedure is simple, so we omit the details.
Let
\begin{eqnarray}
u_m\rightarrow u\quad in\ \widetilde{L}^\infty(\dot{B}_{2,1}^{\frac{\mathbb{N}}{2}-1}),\nonumber\\
\tau_m\rightarrow\tau\quad in\ \widetilde{L}^\infty(\dot{B}_{2,1}^{\frac{\mathbb{N}}{2}}).\nonumber
\end{eqnarray}
Then $(u,\ \tau)$ is a solution of (\ref{prob1c})-(\ref{inic1}) on $(0,\infty)\times R^{\mathbb{N}}$ satisfying (\ref{e2}). By Lemma \ref{lemma3}, we can get $(u,\ \tau)\in C([0,\infty);\dot{B}_{2,1}^{\frac{\mathbb{N}}{2}-1})\times C([0,\infty);\dot{B}_{2,1}^{\frac{\mathbb{N}}{2}})$.\\
\noindent {\bf Third step: Uniqueness}\\
Let $(u,\ \tau)$, $(v,\ n)$ be two solutions of (\ref{prob1c})-(\ref{inic1}). Denote $w=u-v$, $\Gamma=\tau-n$. By using the procedure of proving $\{(u_m,\ \tau_m)|\ m=1,2,\cdots\}$ is a Cauchy sequence in the second step on small interval $[0,\delta]$(some small constant $\delta>0$), we can prove $w=0,\ \Gamma=0$ on $[0,\delta]$. By repeating the procedure, we can obtain $w=0,\ \Gamma=0$ on $[0,T](or\ [0,\infty))$. We finish the proof of Theorem \ref{th3}.\\

\noindent{\bf Proof of Theorem \ref{th4}.}\\
We note that $(u,\ \tau)\in C([0,T];\dot{B}_{2,1}^{\frac{\mathbb{N}}{2}-1})\times C([0,T];\dot{B}_{2,1}^{\frac{\mathbb{N}}{2}})$. By theorem \ref{th3}, we can find $T_1$ such that $T_1>T$ and $(u,\ \tau)$ is a solution on $[0,T_1]\times R^\mathbb{N}$. Since (\ref{serrin}) holds, for any small $\varepsilon_0>0$, we can find $\delta>0$ which depends only on $T'$, such that
\begin{eqnarray}
\|u\|_{\widetilde{L}_{(T_2,T_2+\delta)}^{\rho_1}(\dot{B}_{\infty,\infty}^{-1+\frac{2}{\rho_1}})}+ \|d\|_{\widetilde{L}_{(T_2,T_2+\delta)}^{\rho_2}(\dot{B}_{\infty,\infty}^{\frac{2}{\rho_2}})}+
\|d\|_{\widetilde{L}_{(T_2,T_2+\delta)}^{\rho_3}(\dot{B}_{2,\infty}^{\frac{\mathbb{N}}{2}+\frac{2}{\rho_3}})}\leq\varepsilon_0.\label{last}
\end{eqnarray}
Here we should guarantee $[T_2,T_2+\delta]\in [T,T']$. Let $T_2$ be the initial time. Then by using Lemma \ref{lemma3} and the estimates in the second step of the proof of Theorem \ref{th3}, we can get
\begin{eqnarray}
\|u\|_{\widetilde{L}_{(T_2,T_2+\delta)}^\infty(\dot{B}_{2,1}^{\frac{\mathbb{N}}{2}-1})}+ \|\tau\|_{\widetilde{L}_{(T_2,T_2+\delta)}^\infty(\dot{B}_{2,1}^{\frac{\mathbb{N}}{2}})}\leq \left(\|u(T_2,\cdot)\|_{\dot{B}_{2,1}^{\frac{\mathbb{N}}{2}-1}}
+\|\tau(T_2,\cdot)\|_{\dot{B}_{2,1}^{\frac{\mathbb{N}}{2}}}\right)\nonumber\\
+\theta\left(\|u\|_{\widetilde{L}_{(T_2,T_2+\delta)}^\infty(\dot{B}_{2,1}^{\frac{\mathbb{N}}{2}-1})}+ \|\tau\|_{\widetilde{L}_{(T_2,T_2+\delta)}^\infty(\dot{B}_{2,1}^{\frac{\mathbb{N}}{2}})}\right),\nonumber
\end{eqnarray}
where $0<\theta<1$, and providing $\varepsilon_0$ is small enough. Then we get
\begin{eqnarray}
\|u\|_{\widetilde{L}_{(T_2,T_2+\delta)}^\infty(\dot{B}_{2,1}^{\frac{\mathbb{N}}{2}-1})}+ \|\tau\|_{\widetilde{L}_{(T_2,T_2+\delta)}^\infty(\dot{B}_{2,1}^{\frac{\mathbb{N}}{2}})}\leq C\left(\|u(T_2,\cdot)\|_{\dot{B}_{2,1}^{\frac{\mathbb{N}}{2}-1}}
+\|\tau(T_2,\cdot)\|_{\dot{B}_{2,1}^{\frac{\mathbb{N}}{2}}}\right),\label{s1}
\end{eqnarray}
where the constant $C$ is independent of $T_2$. Let's repeat above procedure, we claim that $T_1\geq T'$. Indeed, if it is not true, we can find $\widetilde{T}\in [T_1,T']$, such that($\widetilde{T}$ is the first blow-up time)
\begin{eqnarray}
\lim_{t\rightarrow\widetilde{T}}\left(\|u(t,\cdot)\|_{\dot{B}_{2,1}^{\frac{\mathbb{N}}{2}-1}}
+\|\tau(t,\cdot)\|_{\dot{B}_{2,1}^{\frac{\mathbb{N}}{2}}}\right)=\infty,\label{s2}
\end{eqnarray}
and $\left(\|u(\widetilde{T}-\frac{\delta}{2},\cdot)\|_{\dot{B}_{2,1}^{\frac{\mathbb{N}}{2}-1}}
+\|\tau(\widetilde{T}-\frac{\delta}{2},\cdot)\|_{\dot{B}_{2,1}^{\frac{\mathbb{N}}{2}}}\right)$ is bounded. But we find that (\ref{s1}) with $T_2$ replaced by $\widetilde{T}-\frac{\delta}{2}$ is contradicting with (\ref{s2}). So we get $T_1\geq T'$, and finish the proof of Theorem \ref{th4}.

\end{document}